\documentclass[12pt,twoside]{amsart}
\usepackage{amsmath}
\usepackage{amsthm}
\usepackage{amsfonts}
\usepackage{amssymb}
\usepackage{latexsym}
\usepackage{mathrsfs}
\usepackage{amsmath}
\usepackage{amsthm}
\usepackage{amsfonts}
\usepackage{amssymb}
\usepackage{latexsym}
\usepackage{geometry}
\usepackage{dsfont}
\usepackage[dvips]{graphicx}
\usepackage{color}
\usepackage[all]{xy}

\date{}
\allowdisplaybreaks[4] \footskip=15pt
\renewcommand{\uppercasenonmath}[1]{}

\usepackage{graphicx,amssymb}
\usepackage[all]{xy}
\usepackage{amsmath}

\allowdisplaybreaks
\usepackage{amsthm}
\usepackage{color}

\theoremstyle{plain}

\theoremstyle{plain}
\newtheorem{theorem}{Theorem}[section]
\newtheorem{proposition}[theorem]{Proposition}
\newtheorem{lemma}[theorem]{Lemma}
\newtheorem{corollary}[theorem]{Corollary}

\newtheorem*{open question}{Open Question}
\newtheorem{definition}[theorem]{Definition}

\theoremstyle{definition}
\newtheorem*{acknowledgement}{Acknowledgement}

\theoremstyle{remark}
\newtheorem{remark}[theorem]{Remark}

\newcommand{\Tor}{\mbox{\rm Tor}}

\newcommand{\Prufer}{Pr\"{u}fer}

\newcommand{\Q}{\mathcal{Q}}

\newcommand{\Id}{\mathrm{Id}}

\def\fd{{\rm fd}}

\def\cwd{{\rm w.gl.dim}}

\def\GV{{\rm GV}}
\def\tor{{\rm tor}}
\def\Hom{{\rm Hom}}
\def\Ext{{\rm Ext}}
\def\Tor{{\rm Tor}}

\def\fPD{{\rm fPD}}

\def\ker{{\rm ker}}
\def\Ker{{\rm Ker}}

\def\Im{{\rm Im}}
\def\Cok{{\rm Cok}}

\def\Nil{{\rm Nil}}

\def\GV{{\rm GV}}
\def\Max{{\rm Max}}
\def\DW{{\rm DW}}
\def\WQ{{\rm WQ}}
\def\T{{\rm T}}
\def\E{{\rm E}}
\def\VN{{\rm VN}}
\def\DQ{{\rm DQ}}
\def\Spec{{\rm Spec}}
\newcommand{\m}{\frak{m}}
\newcommand{\p}{\frak{p}}
\newcommand{\q}{\frak{q}}
\def\Min{{\rm Min}}
\begin{document}
\begin{center}
{\large  \bf  On $\tau_q$-weak global dimensions of commuative rings}

\vspace{0.5cm} Xiaolei Zhang, Ning Bian, Refat Abdelmawla Khaled Assaad, Wei Qi  \\
%\bigskip
\address{Xiaolei Zhang}
\address{School of Mathematics and Statistics, Shandong University of Technology, Zibo 255049, China }
\email{zxlrghj@163.com}

\address{Ning Bian}
\address{School of Mathematics and Statistics, Shandong University of Technology, Zibo 255049, China}
\email{ningbian@sdut.edu.cn}

\address{Refat Abdelmawla Khaled Assaad}
\address{Department of Mathematics, Faculty of Science, University Moulay Ismail Meknes, Box 11201, Zitoune, Morocco}
\email{refat90@hotmail.com}

\address{Wei Qi (Corresponding Author)}
\address{School of Mathematics and Statistics, Shandong University of Technology, Zibo 255049, China}
\email{qwrghj@126.com}
\end{center}
%\begin{figure}[b]
%\rule[-2.5truemm]{5cm}{0.1truemm}\\[2mm]
%{\small }
%\end{figure}

%\begin{figure}[b]
%\rule[-2.5truemm]{5cm}{0.1truemm}\\[2mm]
%{\small }
%\end{figure}
\bigskip
\centerline { \bf  Abstract}
\bigskip
\leftskip10truemm \rightskip10truemm \noindent

In this paper, the $\tau_q$-weak global dimension $\tau_q$-\cwd$(R)$ of a commutative ring $R$ is introduced. Rings  with $\tau_q$-weak global dimension equal to $0$ are studied in terms of homologies,  direct products, polynomial extensions and amalgamations. Besides, we investigate the $\tau_q$-weak global dimensions of polynomial rings.
\vbox to 0.3cm{}

\leftskip0truemm \rightskip0truemm
\bigskip
%\section { \bf Introduction    }
%\bigskip

\section{Introduction}
Throughout this paper, we always assume $R$ is a commutative ring with identity. For a ring $R$, we denote by $\T(R)$ the total quotient ring of $R$, $\Min(R)$ the set of all minimal primes of $R$, $\Nil(R)$ the nil radical of $R$ and $\fPD(R)$ the small finitistic dimension of $R$.

To build a connection of some non-Noetherian domains  with classical domains, Wang et al. \cite{fm97} introduced the notions of $w$-operations on domains, which was generalized to commutative rings with zero-divisors by  Yin et al. \cite{ywzc11}. This makes it possible to study modules and their homological dimensions over commutative rings in terms of $w$-operations. In 2015,  Kim et al.  \cite{KW14} extended the classical notion of flat modules to that of $w$-flat modules. And then Wang et al. \cite{fk14} showed that a ring $R$ is a von Neumann regular ring if and only if every $R$-module is $w$-flat. Later, Wang et al. \cite{WQ15} constructed the $w$-weak global dimensions of commutative rings, and proved that von Neumann regular rings are precisely rings with  $w$-weak global dimensions 0. Also, PvMDs are exactly integral domains with $w$-weak global dimensions at most $1$. They also established a connection between  $w$-weak global dimensions of rings, $w$-weak global dimensions of polynomial rings and the weak global dimensions of its Nagata rings.

For a more detailed study of commutative rings with zero-divisors,  Zhou et al. \cite{ZDC20} introduced the notion of $q$-operations by utilizing finitely generated semi-regular ideals.  $q$-operations are semi-star operations which are weaker than $w$-operations.  The authors in \cite{ZDC20} also proposed $\tau_q$-Noetherian rings (i.e. a ring in which any ideal is $\tau_q$-finitely generated) and study them via module-theoretic point of view, such as $\tau_q$-analogue of the Hilbert basis theorem, Krull's principal ideal theorem, Cartan-Eilenberg-Bass theorem and Krull intersection theorem. Recently, the authors this paper \cite{zq23} introduced and studied the notions of $\tau_q$-flat modules,  $\tau_q$-von Neumann regular rings and  $\tau_q$-coherent rings. The  authors \cite{zq23} showed that a ring $R$ is  $\tau_q$-von Neumann regular ring, if and only if $\T(R[x])$ is a von Neumann regular ring if and only if $R$ is a reduced ring with $\Min(R)$ compact.  The  authors \cite{zq23} also gave the Chase Theorem for $\tau_q$-coherent rings. The main motivation of this paper is to introduce and study the $\tau_q$-flat dimensions of modules and  $\tau_q$-weak global dimensions of rings. We also establish a connection between the $\tau_q$-weak global dimensions of a ring $R$, the $\tau_q$-weak global dimensions of $R[x]$, and the weak global dimensions of $\T(R[x])$.

\section{Preliminary}

In this section, we recall some basic notions on $q$-operations. For more details, refer to  \cite{zq23,ZDC20}. Let $R$ be a ring and $A,B\subseteq R$. Denote by $(A:_RB):=\{r\in R\mid Br\subseteq A\}$.
Recall that an ideal $I$ of $R$ is said to be \emph{dense} if $(0:_RI)=0$; \emph{semi-regular} if there exists a finitely generated dense sub-ideal of $I$; and \emph{regular} if it contains a non-zero-divisor. The set of all finitely generated semi-regular ideals of $R$ is denoted by $\Q(R)$ (or $\Q$ if $R$ is clear). It is well-known that a finitely generated ideal $I=\langle a_0,a_1,\cdots,a_n\rangle$ is semi-regular if and only if the polynomial $f(x)=a_0+a_1x+\cdots+a_nx^n$ is a regular element in $R[x]$ (see \cite[Exercise 6.5]{fk16}).  Lucas \cite{L89} introduced the ring of finite fractions of $R$:
$$Q_0(R):=\{\alpha\in T(R[x])\mid\ \mbox{there exists}\ I\in \Q(R)\ \mbox{such that } I\alpha\subseteq R\}.$$
Note that for any commutative ring $R$, we have $R\subseteq T(R)\subseteq Q_0(R)$.

Let $M$ be an $R$-module. Denote by
\begin{center}
	{\rm $\tor_{\Q}(M):=\{x\in M|Ix=0$, for some $I\in \Q(R) \}.$}
\end{center}
Recall from \cite{wzcc20} that an $R$-module $M$ is said to be \emph{$\Q$-torsion} (resp., \emph{$\Q$-torsion-free}) if $\tor_{\Q}(M)=M$ (resp., $\tor_{\Q}(M)=0$).  A $\Q$-torsion-free module $M$ is called a \emph{Lucas module} if $\Ext_R^1(R/I,M)=0$ for any $I\in \Q$, and the \emph{Lucas envelope} of $M$ is given by
\begin{center}
	{\rm $M_q:=\{x\in \E_R(M)|Ix\subseteq M$, for some $I\in \Q(R)\},$}
\end{center}
where $\E_R(M)$ is the injective envelope of $M$ as an $R$-module.
By  \cite[Theorem 2.11]{wzcc20}, $M_q=\{x\in \T(M[x])|Ix\subseteq M$, for some $I\in \Q(R)\}.$
Obviously, $M$ is a Lucas module if and only if $M_q=M$. A \emph{$\DQ$ ring} $R$ is a ring for which every $R$-module is a Lucas module.
By \cite[Proposition 2.2]{fkxs20}, $\DQ$ rings are exactly rings with small finitistic dimensions equal to $0$.
Recall from \cite{ZDC20} that an submodule $N$ of a  $\Q$-torsion free module $M$ is called a \emph{$q$-submodule} if $N_q\cap M=N$. If an ideal $I$ of $R$ is a $q$-submodule of $R$, then $I$ is also called a \emph{$q$-ideal} of $R$. A \emph{maximal $q$-ideal} is an ideal of $R$ which is maximal among the $q$-submodules of $R$. The set of all maximal $q$-ideals is denoted by $q$-$\Max(R)$, and it is the set of all maximal non-semi-regular ideals of $R$, and thus is non-empty and a subset of $\Spec(R)$ (see
\cite[Proposition 2.5, Proposition 2.7]{ZDC20}).

An $R$-homomorphism $f:M\rightarrow N$ is called to be  a \emph{$\tau_q$-monomorphism} (resp., \emph{$\tau_q$-epimorphism}, \emph{$\tau_q$-isomorphism}) provided that $f_\m:M_\m\rightarrow N_\m$ is a monomorphism (resp., an epimorphism, an isomorphism) over $R_\m$ for any $\m\in q$-$\Max(R)$. By \cite[Proposition 2.7(5)]{ZDC20}, an $R$-homomorphism $f:M\rightarrow N$ is a $\tau_q$-monomorphism (resp., $\tau_q$-epimorphism, $\tau_q$-isomorphism) if  and only if $\Ker(f)$ is  (resp.,  $\Cok(f)$ is, both $\Ker(f)$ and $\Cok(f)$ are) $\Q$-torsion. A sequence of $R$-modules $A\xrightarrow{f} B\xrightarrow{g} C$ is said to be  \emph{$\tau_q$-exact} provided that $A_\m\rightarrow B_\m\rightarrow C_\m$ is  exact as  $R_\m$-modules for any  $\m\in q$-$\Max(R)$. Let $M$ be an $R$-module.   $M$ is said to be \emph{$\tau_q$-finitely generated} provided that there exists  a $\tau_q$-exact sequence $F\rightarrow M\rightarrow 0$ with $F$ finitely generated free.  $M$ is said to be \emph{$\tau_q$-finitely presented} provided that there exists a $\tau_q$-exact sequence $ F_1\rightarrow F_0\rightarrow M\rightarrow 0$ such that $F_0$ and $F_1$ are finitely generated free modules.

Recall from \cite{fk16} a finitely generated ideal   $J$  of $R$ is called a \emph{Glaz-Vasconcelos ideal} (\emph{$\GV$-ideal} for short) if the natural homomorphism $R\rightarrow \Hom_R(J,R)$ is an isomorphism, and the set of all $\GV$-ideals is denoted by $\GV(R)$. Trivially, $\{R\}\subseteq\GV(R)\subseteq\Q(R)$.
A ring $R$ is said to be a $\DW$-ring (resp.,  $\WQ$-ring) if $\GV(R)=\{R\}$ (resp., $\GV(R)=\Q(R)$). The notions of $w$-flat modules were introduced by  Kim and Wang \cite{KW14} by using $\GV(R)$-torison theories, which is similar with the following $\tau_q$-flat modules.

Recall from \cite{zq23} that an $R$-module $M$ is said to be a \emph{$\tau_q$-flat} module provided that, for any $\tau_q$-monomorphism  $f: A\rightarrow B$,  $1_M\otimes f:M\otimes_RA\rightarrow M\otimes_R B$ is a $\tau_q$-monomorphism. The class of $\tau_q$-flat  modules is closed under  $\tau_q$-isomorphisms.  A ring $R$ is a $\DQ$-ring (resp., $\WQ$-ring) if and only if any $\tau_q$-flat module is flat (resp., $w$-flat). The following result gives some characterizations of  $\tau_q$-flat modules.

\begin{lemma}\label{w-coh-c-c}\cite[Theorem 4.3]{zq23}
The following statements are equivalent for an $R$-module $M$.
\begin{enumerate}
    \item $M$ is $\tau_q$-flat;
    \item   for any monomorphism  $f: A\rightarrow B$,  $1_M\otimes f:M\otimes_RA\rightarrow M\otimes_R B$ is a $\tau_q$-monomorphism;
    \item  For any $N$, $\Tor^R_1(M,N)$ is $\Q$-torsion;
      \item  For any $N$ and $n\geq 1$, $\Tor^R_n(M,N)$ is $\Q$-torsion;
      \item  For any ideal $I$, the natural homomorphism $M\otimes_RI\rightarrow MI$ is a $\tau_q$-isomorphism;
       \item  For any  finitely generated $(\tau_q$-finitely generated$)$ ideal $I$, the natural homomorphism $M\otimes_RI\rightarrow MI$ is a $\tau_q$-isomorphism;
      % \item  For any ideal $I$,   $\Tor^R_1(R/I,N)$ is $\Q$-torsion;
          \item  For any finitely generated $(\tau_q$-finitely generated$)$ ideal $I$,   $\Tor^R_1(R/I,M)$ is $\Q$-torsion;
          \item $M_\m$ is a flat $R_\m$-module for any $\m\in q$-$\Max(R)$;
        \item $M\otimes_R\T(R[x])$ is a flat  $\T(R[x])$-module.
\end{enumerate}
\end{lemma}

\section{On $\tau_q$-flat dimensions of modules and  $\tau_q$-weak global dimensions of rings}

The  $w$-flat dimension of  a given $R$-module is defined to be the length of shortest of its shortest
$w$-flat $w$-resolution (see \cite{WQ15}). Now, we introduce the notion of $\tau_q$-flat dimensions as follows.
\begin{definition}
Let $R$ be a ring and $M$ an $R$-module, then $\tau_q$-$\fd_R(M)\leq n$ $(\tau_q$-\fd\ abbreviates $\tau_q$-flat
dimension$)$ if there is a $\tau_q$-exact sequence of $R$-modules
 $$0 \rightarrow  F_n \rightarrow \cdots\rightarrow  F_1 \rightarrow  F_0 \rightarrow  M \rightarrow  0 \ \ \ \ \ \ (\lozenge)$$
with each $F_i$  $\tau_q$-flat. The $\tau_q$-exact sequence $(\lozenge)$ is called a $\tau_q$-flat $\tau_q$-resolution of length n of M. If no such finite $\tau_q$-resolution exists, then
$\tau_q$-$\fd_R(M) =\infty$; otherwise, define $\tau_q$-$\fd_R(M)= n$ if $n$ is the length of a shortest
$\tau_q$-flat $\tau_q$-resolution of $M$.
\end{definition}
It is obvious that an $R$-module $M$ is $\tau_q$-flat if and only if $\tau_q$-$\fd_R(M) = 0$.
 If we denote $\fd_R(M)$ (resp., $w$-$\fd_R(M)$) the flat (resp., $w$-flat) dimension of $M$, then
\begin{center}
$\tau_q$-$\fd_R(M)\leq w$-$\fd_R(M) \leq \fd_R(M)$.
\end{center}
Note that $\tau_q$-$\fd_R(M)= w$-$\fd_R(M)$ if $R$ is a $\WQ$-ring; $\tau_q$-$\fd_R(M)= \fd_R(M)$ if $R$ is a $\DQ$-ring.

%\begin{proposition} Let $0\rightarrow  A\rightarrow B\rightarrow  C \rightarrow  0 $ a $\tau_q$-exact sequence of $R$-modules and  $M$ an $R$-module. Then there exists a long $\tau_q$-exact sequence
% $$\cdots\rightarrow\Tor^R_{n+1}(C, M)\rightarrow  \Tor^R_{n} (A, M)\rightarrow  \Tor^R_{n} (B, M)\rightarrow  \Tor^R_{n} (B, C) $$
% $$\rightarrow\cdots\rightarrow A\otimes_RM\rightarrow  B\otimes_RM\rightarrow  C\otimes_RM\rightarrow 0.$$
%\end{proposition}
%\begin{proof}

%\end{proof}

\begin{lemma}\label{big-Tor} Let $N$ be an $R$-module and $0\rightarrow  A\rightarrow  F\rightarrow  C \rightarrow  0 $ a $\tau_q$-exact sequence of $R$-modules with $F$ a $\tau_q$-flat module. Then for any integer $n>0$,  $\Tor^R_{n+1}(C, N)$
is $\Q$-torsion if and only if so is $\Tor^R_{n} (A, N)$.
\end{lemma}
\begin{proof}  Let $\m$ be a maximal $q$-ideal of $R$. Then $0\rightarrow  A_\m\rightarrow  F_\m\rightarrow  C_\m \rightarrow  0 $ is a short exact sequence of $R_\m$-modules. Hence we have an $R_\m$-exact sequence  $$\Tor_{R_\m}^{n+1} (F_\m, M_\m)\rightarrow\Tor_{R_\m}^{n+1}(C_\m, M_\m)\rightarrow  \Tor_{R_\m}^{n} (A_\m, M_\m)\rightarrow  \Tor_{R_\m}^{n} (F_\m, M_\m).$$
Since $F$ is $\tau_q$-flat, $F_\m$ is a flat $R_\m$-module. It follows that we have natural isomorphisms $$\Tor_{R}^{n+1}(C, M)_\m\cong \Tor_{R_\m}^{n+1}(C_\m, M_\m)\cong \Tor_{R_\m}^{n} (A_\m, M_\m)\cong \Tor_{R}^{n} (A, M)_\m.$$
Consequently, $\Tor^R_{n+1}(C, N)$
	is $\Q$-torsion if and only if so is $\Tor^R_{n} (A, N)$.
\end{proof}

\begin{proposition}\label{w-phi-flat d} Let $n$ be a non-negative integer. Then the following are equivalent for an $R$-module $M.$
\begin{enumerate}
\item $\tau_q$-$\fd_R(M)\leq n$.
\item $\Tor^R_{n+k}(M, N)$ is $\Q$-torsion for all $R$-modules $N$ and all $k > 0$.
\item $\Tor^R_{n+1}(M, N)$ is $\Q$-torsion for all $R$-modules $N$.
\item If $0 \rightarrow  F_n \rightarrow  F_{n-1} \rightarrow  \cdots \rightarrow  F_1 \rightarrow  F_0 \rightarrow  M \rightarrow  0$ is a $\tau_q$-exact sequence,
where $F_0, F_1, \dots , F_{n-1}$ are $\tau_q$-flat $R$-modules, then $F_n$ is $\tau_q$-flat.
\item If $0 \rightarrow  F_n \rightarrow  F_{n-1} \rightarrow  \cdots \rightarrow  F_1 \rightarrow  F_0 \rightarrow  M \rightarrow  0$ is a $\tau_q$-exact sequence,
where $F_0, F_1, \dots , F_{n-1}$ are flat $R$-modules, then $F_n$ is $\tau_q$-flat.
\item If $0 \rightarrow  F_n \rightarrow  F_{n-1} \rightarrow  \cdots \rightarrow  F_1 \rightarrow  F_0 \rightarrow  M \rightarrow  0$ is an exact sequence, where $F_0, F_1, \dots , F_{n-1}$ are $\tau_q$-flat $R$-modules, then $F_n$ is $\tau_q$-flat.
\item If $0 \rightarrow  F_n \rightarrow  F_{n-1} \rightarrow  \cdots \rightarrow  F_1 \rightarrow  F_0 \rightarrow  M \rightarrow  0$ is an exact sequence, where $F_0, F_1, \dots , F_{n-1}$ are flat $R$-modules, then $F_n$ is $\tau_q$-flat.
\item There is an exact sequence $0 \rightarrow  F_n \rightarrow  F_{n-1} \rightarrow  \cdots \rightarrow  F_1 \rightarrow  F_0 \rightarrow  M \rightarrow  0$ where $F_0, F_1, \dots , F_{n-1}$ are flat $R$-modules and $F_n$ is $\tau_q$-flat.
\item There is a  $\tau_q$-exact sequence $0 \rightarrow  F_n \rightarrow  F_{n-1} \rightarrow  \cdots \rightarrow  F_1 \rightarrow  F_0 \rightarrow  M \rightarrow  0$ where $F_0, F_1, \dots , F_{n-1}$ are flat $R$-modules and $F_n$ is $\tau_q$-flat.
\end{enumerate}
\end{proposition}
\begin{proof}
$(2) \Rightarrow(3)$, $(4)\Rightarrow(5)\Rightarrow(7)$, $(4)\Rightarrow(6)\Rightarrow(7)$ and $(8) \Rightarrow(9)\Rightarrow(1)$:  Trivial.	
	
$(1) \Rightarrow(2)$: We will prove $(2)$ by induction on $n$. Suppose $n = 0$. Then
(2) holds by Lemma \ref{w-coh-c-c} as $M$ is $\tau_q$-flat. If $n>0$, then
there is a $\tau_q$-exact sequence  $0 \rightarrow F_n \rightarrow ...\rightarrow F_1\rightarrow F_0\rightarrow M\rightarrow 0$,
where each $F_i$ is  $\tau_q$-flat for $i=0,...,n-1$ and $F_n$ is $\tau_q$-flat. Set $K_0 = \ker(F_0\rightarrow M)$. Then both
$0 \rightarrow  K_0 \rightarrow  F_0 \rightarrow  M \rightarrow  0 $ and $0 \rightarrow  F_n \rightarrow  F_{n-1} \rightarrow...\rightarrow  F_1 \rightarrow  K_0 \rightarrow  0$ are $\tau_q$-exact, and $\tau_q$-fd$_R(K_0)\leq n-1$. By induction, $\Tor^R_{n-1+k}(K_0, N)$ is $\Q$-torsion
for all  $R$-modules $N$ and all $k > 0$. Thus, it follows from Lemma \ref{big-Tor} that $\
\Tor^R_{n+k}(M, N)$ is $\Q$-torsion.

$(3)\Rightarrow(4)$  Set $L_n = F_n$ and $L_i = \Im(F_i \rightarrow  F_{i-1})$, where $i = 1, \dots, n - 1$.
Then both $0 \rightarrow  L_{i+1} \rightarrow  F_i \rightarrow  L_i \rightarrow  0$ and $0 \rightarrow  L_1 \rightarrow  F_0 \rightarrow  M \rightarrow  0$ are $\tau_q$-exact sequences. By using Lemma \ref{big-Tor} repeatedly, we can obtain that $\Tor^R_1 (F_n, N)$ is
$\Q$-torsion for all $R$-modules $N$. Thus $F_n$ is $\tau_q$-flat.

$(7) \Rightarrow(8)$: Since every $R$-module has a flat cover, we can induce a long exact sequence  $$\cdots \rightarrow  F'_n \rightarrow  F_{n-1} \xrightarrow{d_{n-1}}  F_{n-2} \rightarrow  \cdots \rightarrow  F_1 \rightarrow  F_0 \rightarrow  M \rightarrow  0$$ with each term flat. Setting $F_n=\Ker(d_{n-1})$, we have $F_n$ is $\tau_q$-flat by $(7)$.
\end{proof}

\begin{proposition}\label{local-q-fd} Let $M$ be an $R$-module. Then
\begin{enumerate}
\item $\tau_q$-$\fd_R(M)\leq n$ if and only if $\fd_{R_\m}(M_\m)\leq n$  for all $\m\in q$-$\Max(R)$.
\item  $\tau_q$-$\fd_R(M)=\sup\{\fd_{R_\m}(M_\m) \mid \m \in \tau_q$-$\Max(R)\}$.
\end{enumerate}
\end{proposition}
\begin{proof} (1) Suppose  $\tau_q$-$\fd_R(M)\leq n$. Let  $\m\in  q$-$\Max(R)$ and $N$ an $R_\m$-module. Then  $\Tor^R_{n+1}(M, N)$ is $\Q$-torsion, and so $\Tor_{R_\m}^{n+1}(M_\m, N)=0$. Hence  $\fd_{R_\m}(M_\m)\leq n$. On the other hand, let $N$ be an $R$-module.  For any $\m\in q$-$\Max(R)$, we have $0=\Tor_{R_\m}^{n+1}(M_\m, N_\m)\cong\Tor^R_{n+1}(M, N)_\m$. So $\Tor^R_{n+1}(M, N)$ is $\Q$-torsion, and hence $\tau_q$-$\fd_R(M)\leq n$ by Proposition \ref{w-phi-flat d}.

 (2) It follows by (1).
\end{proof}

\begin{corollary}\label{qfd-exac} Let $0\rightarrow A\rightarrow B\rightarrow C\rightarrow 0$ be a $\tau_q$-short exact sequence of $R$-modules \cite[Theorem 3.6.7]{fk16}. 
\begin{enumerate}
		\item Then $\tau_q$-$\fd_R(C)\leq 1+\max\{\tau_q$-$\fd_R(A),\tau_q$-$\fd_R(B)\} $.
		\item  If $\tau_q$-$\fd_R(B)<\tau_q$-$\fd_R(C)$, then 
		$\tau_q$-$\fd_R(A)=\tau_q$-$\fd_R(C)-1\geq \tau_q$-$\fd_R(B)$.
\end{enumerate}
\end{corollary}
\begin{proof} Since $0\rightarrow A\rightarrow B\rightarrow C\rightarrow 0$ is a $\tau_q$-short exact sequence, 
$0\rightarrow A_\m\rightarrow B_\m\rightarrow C_\m\rightarrow 0$ is a short exact sequence of $R_\m$-modules for any $\m\in q$-$\Max(R)$. So the results follows by Proposition \ref{local-q-fd} and \cite[Theorem 3.6.7]{fk16}.
\end{proof}
\begin{corollary}\label{qfd-ds} Let $\{M_i\mid i\in \Gamma\}$ be a family of  $R$-modules.  Then
	\begin{center}
		$\tau_q$-$\fd_R(\bigoplus\limits_{ i\in \Gamma}M_i)=\sup\{\tau_q$-$\fd_R(M_i)\mid i\in \Gamma\}.$
	\end{center}	
\end{corollary}
\begin{proof} Note that $(\bigoplus\limits_{ i\in \Gamma}M_i)_\m\cong \bigoplus\limits_{ i\in \Gamma}(M_i)_\m$ for any $\m\in q$-$\Max(R)$. So the result follows by Proposition \ref{local-q-fd}.
\end{proof}
\begin{corollary}\label{qfd-q-iso} Let $M$ and $N$ be $R$-modules such that $M$ is $\tau_q$-isomorphic to $N$. Then $\tau_q$-$\fd_R(M)=\tau_q$-$\fd_R(N)$.
\end{corollary} 
\begin{proof} Since $M$ is $\tau_q$-isomorphic to $N$, then there is an exact sequence of $R$-modules  $0\rightarrow K_1\rightarrow M\rightarrow N\rightarrow C_1\rightarrow 0$ or $0\rightarrow K_2\rightarrow N\rightarrow M\rightarrow C_2\rightarrow 0$ such that $K_i$ and $C_i$ are $\Q$-torsion $(i=1,2)$. So $M_\m$ isomorphic to $N_\m$ for any $\m\in q$-$\Max(R)$. 
Hence $\tau_q$-$\fd_R(M)=\tau_q$-$\fd_R(N)$ by Proposition \ref{local-q-fd}.
\end{proof}

The author in \cite{WQ15} also introduce the weak global dimension of a given ring. We can similarly introduce the $\tau_q$-weak global dimension of a ring $R$.
\begin{definition} The $\tau_q$-weak global dimension of a ring $R$ is defined by
\begin{center}
$\tau_q$-\cwd$(R) = \sup\{\tau_q$-$\fd_R(M)\mid M$ is an $R$-module$\}.$
\end{center}
\end{definition}
Obviously, if we denote the weak (resp., $w$-weak) global dimension of a ring $R$ by
\cwd$(R)$ (resp., $w$-\cwd$(R)$), then
\begin{center}
$\tau_q$-$\cwd(R)\leq w$-$\cwd(R) \leq \cwd(R)$.
\end{center}
Note that $\tau_q$-$\cwd(R)= w$-$\cwd(R)$ if $R$ is a $\WQ$-ring; $\tau_q$-$\cwd(R)= \cwd(R)$ if $R$ is a $\DQ$-ring.

The following result characterizes rings with $\tau_q$-weak global dimensions at most $n$.
\begin{theorem}\label{q-dim} The following statements are equivalent for $R.$
\begin{enumerate}
\item	$\tau_q$-$\cwd(R)\leq n$.
\item $\tau_q$-\fd$_R(M)\leq n$ for all $R$-modules $M.$
\item  $\Tor^R_{n+k}(M, N)$ is $\Q$-torsion  for all $R$-modules $M$ and $N$ and all $k > 0.$
\item $\Tor^R_{n+1}(M, N)$ is $\Q$-torsion for all $R$-modules $M$ and $N.$
\item $\tau_q$-\fd$_R(R/I)\leq n$ for all ideals $I$ of $R$.
\item $\tau_q$-\fd$_R(R/I)\leq n$ for all $\tau_q$-finitely generated ideals $I$ of $R.$
\item$\tau_q$-\fd$_R(R/I)\leq n$ for all finitely generated ideals $I$ of $R.$
\end{enumerate}
Consequently, the $\tau_q$-weak global dimension of $R$ is also determined by the
formulas:
\begin{align*}
\tau_q\mbox{-}\cwd(R)&=\sup\{\cwd(R_\m)\mid \m \in\tau_q\mbox{-}\Max(R)\} \\
&= \sup \{\tau_q\mbox{-}\fd_R(R/I) |\ I\ \mbox{is\ an\ ideal\ of}\ R\}\\
&= \sup \{\tau_q\mbox{-}\fd_R(R/I) |\ I\ \mbox{is\ a}\ \tau_q\mbox{-}\mbox{finite\ type\ ideal\ of}\ R\} \\
&= \sup \{\tau_q\mbox{-}\fd_R(R/I) |\ I\ \mbox{is\ a\ finitely\ generated\ ideal\ of}\ R\}.
\end{align*}
\end{theorem}
\begin{proof}
	
	$(1) \Leftrightarrow (2)$ and $(2) \Rightarrow  (5) \Rightarrow  (6) \Rightarrow  (7) $:  Trivial.
	
	$(2) \Leftrightarrow  (3)\Leftrightarrow  (4)$: Follows from Proposition \ref{w-phi-flat d}.
	
	$(4)\Rightarrow  (5)$:  Trivial.
	
	$(7) \Rightarrow  (1)$: Let $M$ be an $R$-module and $0 \rightarrow F_n \rightarrow ...\rightarrow F_1\rightarrow F_0\rightarrow M\rightarrow 0$ an exact sequence, where $F_0, F_1, . . . , F_{n-1}$ are flat $R$-modules.
	To complete the proof, it suffices, by Proposition \ref{w-phi-flat d}, to prove that $F_n$ is
	$\tau_q$-flat. Let $I$ be a finitely generated ideal of $R$.
	Thus $\tau_q$-fd$_R(R/I)\leq n$ by (7). It follows from Lemma \ref{big-Tor} that $\Tor^R_1 (R/I, F_n)\cong \Tor^R_{n+1}(R/I, M)$
	is $\Q$-torsion.

The first equality of the consequence  follows by Proposition \ref{local-q-fd}, and the others follow by the above of this Theorem.
\end{proof}

\begin{proposition}\label{fd-r} Let $R=R_1\times R_2\times\cdots R_n$ be a finite direct product of rings. Then
\begin{center}
	 $\tau_q$-$\cwd(R)=\max\{\tau_q$-$\cwd(R_i)\mid i=1,2,\dots,n\}$.
\end{center}
\end{proposition}
\begin{proof}  Note that every (finitely generated semi-regular) ideal of $R$ is of the form $I=I_1\times I_2\times\cdots I_n$ where each $I_i$ is a  (finitely generated semi-regular) ideal of $R_i$. So  $\tau_q$-\fd$_R(R/I)\leq n$ for all ideals $I$ of $R$ if and only if $\tau_q$-\fd$_{R_i}(R_i/I_i)\leq n$ for all ideals $I_i$ of $R_i$ and all $i=1,2,\dots,n$. Hence the result holds by Theorem \ref{q-dim}.
\end{proof}

\section{More results on $\tau_q$-\VN\ regular rings}

Recall from \cite[Definition 4.7]{zq23} that a ring $R$ is said to be a \emph{$\tau_q$-\VN\ regular ring} $($short for  $\tau_q$-von Neumann regular ring$)$ provided that all $R$-modules are $\tau_q$-flat. Certainly, integral domains and  von Neumann regular rings are $\tau_q$-\VN\ regular. The following result gives a homological characterization of $\tau_q$-von Neumann regular rings.
\begin{proposition}\label{char-tvn}
	The following assertions are equivalent for a ring $R$.
	\begin{enumerate}
		\item $\tau_q$-$\cwd(R)=0$;
		\item $R$ is a $\tau_q$-\VN\ regular ring;
		\item  $\Tor^R_{k}(M, N)$ is $\Q$-torsion  for all $R$-modules $M$ and $N$ and all $k > 0.$
		\item $\Tor^R_{1}(M, N)$ is $\Q$-torsion for all $R$-modules $M$ and $N.$
		\item  $R/I$ is $\tau_q$-flat for all $($ $\tau_q$-finitely generated, or finitely generated$)$ ideals $I$ of $R$.
		\item for any finitely generated ideal $K$ of $R$, there exists an ideal $I\in\Q$ such that $IK=K^2$;
		\item $R_{\m}$ is a von Neumann regular ring for any $\m\in q$-$\Max(R)$;
		\item   $\T(R[x])$ is a von Neumann regular ring;
		\item    $R$ is a reduced ring and $\Min(R)$ is compact.
		\item  $Q_0(R)$ is a von Neumann regular ring.
	\end{enumerate}
\end{proposition}
\begin{proof} The equivalences of $(1)-(5)$ follow by Theorem \ref{q-dim}, and the equivalences of $(2)$ and $(6)-(9)$ follow from
\cite[Theorem 4.9]{zq23}.

$(9)\Leftrightarrow (10)$ Note that $R\subseteq Q_0(R)$. So the reducedness of $Q_0(R)$ implies the reducedness of $R$. Hence the equivalence   follows by  \cite[Theorem 7.6]{L05}.
\end{proof}
\begin{remark}
If a ring $R$ satisfies that $\T(R)$ is a von Neumann regular ring, then $R$ is a $\tau_q$-\VN\ regular ring, since the former is equivalent to that  $R$ is a reduced ring, $\Min(R)$ is compact,  and if a finitely generated ideal is contained in the union of the minimal primes of $R$ then it is contained in one of them (see \cite[Proposition 9]{Q71}). The author in \cite{Q71} also gave a counter-example to show the converse does not hold in general.
\end{remark}

\begin{corollary}
	Let $R$ be a ring. Then the following statements are equivalent:
	\begin{enumerate}
		\item $R$ is a von Neumann regular ring;
		\item $R$ is a $\tau_q$-von Neumann regular ring and a $\DQ$-ring;
		\item $R$ is a $\tau_q$-von Neumann regular ring and a $\DW$-ring;
		\item $R$ is a $\tau_q$-von Neumann regular ring and a $\WQ$-ring.
	\end{enumerate}		
\end{corollary}
\begin{proof} $(1)\Rightarrow (2)$ Let $R$ be a von Neumann regular ring. Then $R$ is trivially a $\tau_q$-von Neumann regular ring. Note that every finitely generated ideal of $R$ is generated by an idempotent. So $R$ is a $\DQ$-ring.
	
	$(2)\Rightarrow (3)\Rightarrow (4)$ Trivially.
	
	$(4)\Rightarrow (1)$ Let $M$ be an $R$-module, then $M$ is $\tau_q$-flat since  $R$ is $\tau_q$-von Neumann regular, and hence $w$-flat since  $R$ is a $\WQ$-ring. Consequently,  $R$ is a von Neumann regular ring by \cite[Theorem 4.4]{fk14}.
\end{proof}

\begin{corollary}\label{fd-r-0} Let $R=R_1\times R_2\times\cdots R_n$ be a finite direct product of rings. Then $R$ is a $\tau_q$-\VN\ regular ring if and only if each $R_i$ is a $\tau_q$-\VN\ regular ring $( i=1,2,\dots,n)$.
\end{corollary}
\begin{proof} It follows by Proposition \ref{fd-r} and Proposition \ref{char-tvn}.
\end{proof}

Let $f:A\rightarrow B$ be a ring homomorphism and $J$ an ideal of $B$. Recall from \cite{df09} that the  \emph{amalgamation} of $A$ with $B$ along $J$ with respect to $f$, denoted by $A\bowtie^fJ$, is defined as $$A\bowtie^fJ=\{(a,f(a)+j)|a\in A,j\in J\},$$  which is  a subring of of $A \times B$.  It follows from \cite[Proposition 4.2]{df09} that  $A\bowtie^fJ$  is the pullback $\widehat{f}\times_{B/J}\pi$,
where $\pi:B\rightarrow B/J$ is the natural epimorphism and $\widehat{f}=\pi\circ f$:
$$\xymatrix@R=20pt@C=25pt{
	A\bowtie^fJ\ar[d]^{p_B}\ar[r]_{p_A}& A\ar[d]^{\widehat{f}}\\
	B\ar[r]^{\pi}&B/J. \\
}$$

\begin{lemma}\cite[Proposition 5.4]{df09}\label{am-1}
	Let $f:A\rightarrow B$ be a ring homomorphism and $J$ an ideal of $B$. Then $A\bowtie^fJ$ is a reduced ring if and only if $A$ is reduced and $\Nil(B)\cap J=0$.
\end{lemma}

\begin{lemma}\cite[Proposition 2.6]{df10}\label{am-2}
	Let $f:A\rightarrow B$ be a ring homomorphism and $J$ an ideal of $B$. Let $\p$ be a prime ideal of $A$ and $\q$ a prime ideal of $B$. Set
	\begin{enumerate}
		\item $\p'^f:=\p\bowtie^fJ=\{(p,f(p)+j)\mid p\in\p\}$;
		\item $\overline{\q}^f:=\{(a,f(a)+j)\in A\bowtie^fJ\mid f(a)+j\in\q\}$.
	\end{enumerate}
	Then every prime ideal of $A\bowtie^fJ$ is of the form $\p'^f$ or $\overline{\q}^f$ with $\p\in\Spec(A)$ and $\q\in\Spec(B)-V(J)$.
\end{lemma}

\begin{lemma}\cite[Corollary 2.8]{df16}\label{am-3}
	Let $f:A\rightarrow B$ be a ring homomorphism and $J$ an ideal of $B$.  Set
	$$\mathcal{X}=\bigcup\limits_{\q\in\Spec(B)-V(J)}V(f^{-1}(\q+J)).$$
	Then the following properties hold.
	\begin{enumerate}
		\item The map defined by $\q\mapsto \overline{\q}^f$ establishes a homeomorphism of $\Min(B)-V(J)$ with $\Min(A\bowtie^fJ)-V(\{0\}\times J)$.
		\item  The map defined by $\p\mapsto \p'^f$ establishes a homeomorphism of $\Min(A)-\mathcal{X}$ with $\Min(A\bowtie^fJ)\cap V(\{0\}\times J)$.
	\end{enumerate}
	Therefore, we have
	$$\Min(A\bowtie^fJ)=\{\p'^f\mid\p\in \Min(A)-\mathcal{X}\}\cup \{\overline{\q}^f\mid\q\in \Min(B)-V(J)\}.$$
\end{lemma}

\begin{proposition}\label{am}
	Let $f:A\rightarrow B$ be a ring homomorphism and $J$ an ideal of $B$. Then $A\bowtie^fJ$ is a $\tau_q$-\VN\ regular ring if and only if $A$ is reduced, $\Nil(B)\cap J=0$, and $\{\p'^f\mid\p\in \Min(A)-\mathcal{X}\}$ and $ \{\overline{\q}^f\mid\q\in \Min(B)-V(J)\}$ are compact.
\end{proposition}
\begin{proof}
	It follows by Lemma \ref{am-1}, Lemma \ref{am-3} and Proposition \ref{char-tvn}.
\end{proof}

Recall from \cite{df16} that, by setting  $f=\Id_A: A\rightarrow A$ to be the identity  homomorphism of $A$, we denote by $A \bowtie J:=A \bowtie^{\Id_A}J$ and call it the amalgamated algebra of $A$ along $J$. Recall from \cite[Theorem 2.1]{CM12} that a ring  $A\bowtie J$ is a \VN\ regular ring if and only if $A$ is  a \VN\ regular ring.
From Proposition \ref{am}, one can easily deduce the following result.

\begin{proposition}
	Let $J$ be an ideal of $A$. Then $A\bowtie J$ is a $\tau_q$-\VN\ regular ring if and only if $A$ is  a $\tau_q$-\VN\ regular ring.
\end{proposition}

\begin{corollary}
Let $A$ be an integral domain or a von Neumann regular ring, and $J$ be an ideal of $A$. Then $A\bowtie J$ is a $\tau_q$-\VN\ regular ring.
\end{corollary}

Recall from \cite{zdxq23} that an $R$-module is said to be a semi-regular flat module if $\Tor^R_1(R/I,M)=0$ for any $I\in\Q.$ Trivially, all flat modules are semi-regular flat.

\begin{proposition}\label{sf-vn}
 Suppose every semi-regular flat  $R$-module is flat. Then $R$ is a $\tau_q$-\VN\ regular ring.
\end{proposition}
\begin{proof}
Let $I$ be a finitely generated semi-regular ideal of $R$ and  $M$ an $R$-module. Since $R/I$ is $\Q$-torsion, $R/I\otimes_R\T(R[x])=0$ by [\cite{zq23}, Proposition 2.3]. Hence   $\Tor^R_1(R/I,M\otimes_R\T(R[x]))\cong \Tor^R_1(R/I\otimes_R\T(R[x]),M)=0$. It follows that $M\otimes_R\T(R[x])$ is a semi-regular flat  $R$-module, and hence is a flat $R$-module by hypotheses. So $$M\otimes_R(\T(R[x])\otimes_R\T(R[x]))\cong M\otimes_R(\T(R[x])\otimes_{R[x]}R[x]\otimes_R\T(R[x]))\cong \bigoplus\limits_{i=1}^\infty M\otimes_R\T(R[x])$$ is a flat $\T(R[x])$-module. Hence $M\otimes_R\T(R[x])$ is a flat $\T(R[x])$-module, and so $M$ is a $\tau_q$-flat $R$-module by Lemma \ref{w-coh-c-c}. It follows that $R$ is a $\tau_q$-\VN\ regular ring.
\end{proof}
\begin{remark}
We do not known if the converse of Corollary \ref{sf-vn} is true. And we propose the following conjecture:

\textbf{Conjecture:} A ring  $R$ is a $\tau_q$-\VN\ regular ring if and only if every semi-regular flat  $R$-module is flat.
\end{remark}

\section{On $\tau_q$-weak global dimensions of polynomial rings}

The following result connect the classical weak global dimensions (flat dimsions) and  $\tau_q$-weak global dimensions  ($\tau_q$-flat dimsions).
\begin{proposition}\label{fd-TR} Let $R$ be a ring and $M$ an $R$-module.
	\begin{enumerate}
\item $\tau_q$\mbox{-}\fd$_R(M)=$\fd$_{\T(R[x])}(M\otimes_R\T(R[x])).$
\item $\tau_q$\mbox{-}\cwd$(R)\leq$\cwd$(\T(R[x]))$.
	\end{enumerate}
\end{proposition}
\begin{proof}

(1) Suppose  $\tau_q$\mbox{-}\fd$(M)\leq n$. Then there exists an exact sequence $0\rightarrow P_n\rightarrow \cdots\rightarrow P_1\rightarrow P_0\rightarrow M \rightarrow 0$
where each $P_i$ is $\tau_q$-flat.  By applying $-\otimes_R\T(R[x])$, we have  $0\rightarrow P_n\otimes_R\T(R[x])\rightarrow \cdots\rightarrow P_1\otimes_R\T(R[x])\rightarrow P_0\otimes_R\T(R[x])\rightarrow M\otimes_R\T(R[x]) \rightarrow 0$ is a flat resolution of $(M\otimes_R\T(R[x]))$. So \fd$_{\T(R[x])}(M\otimes_R\T(R[x]))\leq n$. On the other hand, \fd$_{\T(R[x])}(M\otimes_R\T(R[x]))\leq n$. Then for any $R$-module $N$ we have $\Tor_{n+1}^{\T(R[x])}(M\otimes_R\T(R[x]),N\otimes_R\T(R[x]))\cong \Tor_{n+1}^{R}(M,N)\otimes_R\T(R[x]) =0$. Hence $\Tor_{n+1}^{R}(M,N)$ is $\Q$-torsion by [\cite{zq23}, Proposition 2.3], which implies that   $\tau_q$\mbox{-}\fd$(M)\leq n$.

(2) It follows by (1).
\end{proof}

The following two results give two characterizations of weak global dimensions and $\tau_q$\mbox{-}weak global dimensions under the assumption of  coherence.
\begin{lemma}\cite[Lemma  3.7]{WQ15}\label{cwd-quo} Let $R$ be a coherent ring. Then \cwd$(R)\leq n$ if and only if \fd$_RR/\m\leq n$ for each $\m\in\Max(R)$.
\end{lemma}
Recall from \cite{zq23} that a ring $R$ is said to be  \emph{$\tau_q$-coherent} provided that every $\tau_q$-finitely generated ideal of $R$ is $\tau_q$-finitely presented, or equivalently every finitely generated ideal of $R$ is $\tau_q$-finitely presented.

\begin{lemma}\label{q-cwd-quo} Let $R$ be a $\tau_q$-coherent ring. Then $\tau_q$\mbox{-}\cwd$(R)\leq n$ if and only if $\tau_q$\mbox{-}\fd$_RR/\m\leq n$ for each $\m\in q$-$\Max(R)$.
\end{lemma}
\begin{proof} It is similar with the proof of \cite[Proposition 3.8]{WQ15}.
\end{proof}

\begin{proposition}\label{T-q-coh} Let $R$ be a ring with  $\T(R[x])$  coherent. Then $R$ is a $\tau_q$-coherent ring.
\end{proposition}
\begin{proof}
 Let $I$ be a finitely generated ideal of $R$. Then  $I\otimes_R \T(R[x])$ is a finitely generated ideal of $\T(R[x])$. Since $\T(R[x])$ is a coherent ring,  $I\otimes_R \T(R[x])$ is a finitely presented ideal of $\T(R[x])$. And hence $I$ is is $\tau_q$-finitely presented by \cite[Theorem 3.3]{zq23}.
\end{proof}

\begin{theorem}\label{q-cwd-cdw} Let $R$ be a ring with  $\T(R[x])$  coherent.  Then
	\begin{center}
 $\tau_q$\mbox{-}\cwd$(R)=$\cwd$(\T(R[x]))$.	\end{center}	
Consequently,	 $\tau_q$\mbox{-}\cwd$(R)=0$ or $\infty$.
\end{theorem}
\begin{proof} Let $n$ be a positive integer. It follows by Lemma \ref{cwd-quo}, Lemma \ref{q-cwd-quo} and the Proposition \ref{T-q-coh} that we have the following equivalences:
	\begin{align*}
		&\tau_q\mbox{-}\cwd(R)\leq n\\
		\Leftrightarrow &\  \tau_q\mbox{-}\fd_R(R/\m)\leq n\ \mbox{for every}\ \m\in q\mbox{-}\Max(R),\\
		\Leftrightarrow &\   \fd_{T(R[x])}(T(R[x])/\m\otimes_{R}T(R[x]))\leq n\  \mbox{for every}\ \m\in q\mbox{-}\Max(R),\\
		\Leftrightarrow &\  \cwd(\T(R[x]))\leq n.
	\end{align*}
Hence $\tau_q$\mbox{-}\cwd$(R)=$\cwd$(\T(R[x]))$.
	
Note that $\T(R[x])$ is a  coherent total ring of quotients. So by \cite[Proposition 6.1]{BG07} \cwd$(\T(R[x]))=0,1,$ or $\infty.$ If \cwd$(\T(R[x]))=1$, then $\T(R[x])$ is a semi-hereditary total ring of quotients. It follows by \cite[Theorem 3.12(i)]{BG07} that $R[x]$ is a semi-hereditary ring.  Then by \cite[Theorem 2]{E61}, \cwd$(\T(R[x]))=0$. Hence  $\tau_q$\mbox{-}\cwd$(R)=$\cwd$(\T(R[x]))=0$ or $\infty$.
\end{proof}

\begin{remark} The $\tau_q$-weak global dimension of a ring can be neither $0$ nor $\infty$ in general. Indeed, let $S=\prod\limits_{i=1}^\infty\mathbb{Q}[x]$ be the ring of countably infinite copies of  products of polynomial ring $\mathbb{Q}[x]$ with coefficients in rational field $\mathbb{Q}$.  let $R$  be the subring of $S$ that generated by the sequence $(x,0,x^2,\cdots)$ and all sequences that eventually consist of constants.  Then $R$ is a ring with weak global dimension equal to $1$ but not semi-hereditary (see \cite[Page 54]{G89}). So  $\tau_q$\mbox{-}\cwd$(R)\leq$ \cwd$(R)=1.$ Assume that $\tau_q$\mbox{-}\cwd$(R)=0.$ Let  $K$ be the ideal generated by $(x,0,x^2,\cdots)$. Then there is an ideal $I\in\Q$ such that $IK=K^2$.  Comparing the components of $I(x,0,x^2,\cdots)=IK=K^2=(x^2,0,x^4,\cdots)R$. We have each $(2n-1)$-component of $I$ is generated by $x^{n-1}$. This is impossible since there exists an element $r\in I$ such that $r$ is eventually non-zero as $I$ is semi-regular. It follows that  $\tau_q$\mbox{-}\cwd$(R)=1$.

% We claim that $R$ is a $\DQ$-ring, and so  $\tau_q$\mbox{-}\cwd$(R)=1.$Indeed, every maximal ideal of $R$ is generated by a sequence of the form $(1,1,\dots,1,f,1,1,\dots)$ where $f$ is a non-constant irreducible polynomial in $\mathbb{Q}[x]$.
\end{remark}

\begin{lemma}\label{L-poly} Let $R$ be a ring, $\m\in q$-$\Max(R)$. Then
$\T(R[x])_{\m\otimes_R\T(R[x])}\cong R[x]_{\m[x]}$.
\end{lemma}
\begin{proof} Let $\sum(\tau_q)$ be the set of all polynomials with contents in $\Q$. Then it is easy to verify $\sum(\tau_q)\subseteq R[x]-\m[x]$, so we have
$$\T(R[x])_{\m\otimes_R\T(R[x])}\cong (R[x]_{\sum(\tau_q)})_{\m[x]_{\sum(\tau_q)}}\cong R[x]_{\m[x]}.$$
\end{proof}

\begin{theorem}\label{poly} Let $R$ be a ring with  $\T(R[x])$ coherent.  Then
	\begin{center}
		$\tau_q$\mbox{-}\cwd$(R[x])=$	$\tau_q$\mbox{-}\cwd$(R)$.	\end{center}
\end{theorem} 		
\begin{proof} Suppose 	$\tau_q$\mbox{-}\cwd$(R[x])\leq n$. Let $M$ be an $R$-module and $\m\in q$-$\Max(R)$. Then $\m[x]\in q$-$\Max(R[x])$. Indeed, let $I$ be a non-semi-regular (equivalently, non-regular) ideal of $R[x]$ that contains $\m[x]$. Then $I$ does not contain a regular element and so its content $c(I)$ is  non-semi-regular that contains $\m$. By the maximality of  $\m$, we have $c(I)=\m$ which implies $I=\m[x]$. It follows by Proposition \ref{local-q-fd} and Lemma \ref{L-poly} that
\begin{center}
	\fd$_{\T(R[x])_{\m\otimes_R\T(R[x])}}M\otimes_R\T(R[x])_{\m\otimes_R\T(R[x])}$=\fd$_{R[x]_{\m[x]}}M[x]_{\m[x]}\leq n.$
\end{center}
Hence \fd$_{\T(R[x])}M\otimes_R\T(R[x])\leq n$, and so $\tau_q$-\fd$_R(M)\leq n$ by Proposition \ref{fd-TR}. Consequently, $\tau_q$\mbox{-}\cwd$(R)\leq n$.

On the other hand, suppose $\tau_q$\mbox{-}\cwd$(R)\leq n$. Let $\p\in q$-$\Max(R[x])$.  Then $\p\cap R$ is also maximal  non-regular. So $\q := \p\cap R\in q$-$\Max(R)$ and $\p = \q[x]$.
It follows by Theorem \ref{q-dim}, Lemma \ref{L-poly} and Theorem \ref{q-cwd-cdw} that
\begin{align*}
	\tau_q\mbox{-\cwd}(R[x])=&\sup\{\mbox{\cwd}(R[x]_\p)\mid \p\in q\mbox{-\Max}(R[x])\}\\
	= & \sup\{\mbox{\cwd}(T(R[x])_{\p\otimes_RT(R[x])})\mid \p\in q\mbox{-\Max}(R[x])\} \\
	\leq & \mbox{\cwd}(\T(R[x]))\\
	=& \tau_q\mbox{-\cwd}(R).
\end{align*}
Consequently, $\tau_q$-\cwd$(R)=\tau_q$-\cwd$(R[x])$.
\end{proof}

\begin{corollary} Let $R$ be a ring.  Then $R$ is a $\tau_q$-\VN\ regular ring if and only if so is $R[x]$.
\end{corollary} 	
\begin{proof} Suppose $R$ is a $\tau_q$-\VN\ regular ring. Then $\T(R[x])$ is a \VN\ regular ring, and so is coherent. Hence $\tau_q$\mbox{-}\cwd$(R[x])=\tau_q$\mbox{-}\cwd$(R)=0$ by Theorem \ref{poly}.  Hence $R[x]$ is a $\tau_q$-\VN\ regular ring. On the other hand, suppose  $R[x]$ is a $\tau_q$-\VN\ regular ring. Then  $\T(R[x,y])$ is a \VN\ regular ring, so is  $\T(R[x])\cong \T(R[x,y])/y\T(R[x,y])$  . Hence $R$ is a $\tau_q$-\VN\ regular ring.
\end{proof}

\begin{acknowledgement}\quad\\
The fourth author was supported by National Natural Science Foundation of China (No. 12201361).
\end{acknowledgement}

\end{document}